\documentclass[10pt]{article}
\usepackage{amsmath}
\usepackage{amsthm}
\usepackage{amsfonts}
\usepackage{eucal}

\usepackage[sc]{mathpazo} 
\usepackage[scaled]{helvet} 
\usepackage{eulervm} 
\usepackage[pdftex]{hyperref}

\theoremstyle{plain}
\newtheorem{lem}{Lemma}
\newtheorem{thm}{Theorem}
\newtheorem{coroll}{Corollary}
\newtheorem{prop}{Proposition}

\theoremstyle{definition}
\newtheorem{remark}{Remark}
\newtheorem{example}{Example}

\def\ztt{\zeta^{t}}
\def\zt{\zeta}
\def\zts{\zeta^{\star}}

\def\C{\mathbf C}
\newcommand{\Ha}[2]{\ensuremath{H_{#1}^{(#2)}}}

\newcommand{\N}{\ensuremath{\mathbb{N}}}
\newcommand{\E}{\ensuremath{\mathbb{E}}}
\newcommand{\fallfak}[2]{\ensuremath{#1^{\underline{#2}}}}
\newcommand{\stir}[2]{\genfrac{ [ }{ ] }{0pt}{}{#1}{#2}}

\DeclareMathOperator{\Exp}{\text{Exp}}

\DeclareMathOperator{\Li}{\text{Li}}

\begin{document}
\title{A Note on Harmonic number identities, Stirling series and multiple zeta values}%
\author{Markus Kuba and Alois Panholzer}
\date{\today}

\maketitle
\begin{abstract}
We study a general type of series and relate special cases of it to Stirling series, infinite series discussed by Choi and Hoffman, and also to special values of the Arakawa-Kaneko zeta function, complementing and generalizing earlier results. 
Moreover, we survey properties of certain truncated multiple zeta and zeta star values, pointing out their relation to finite sums of harmonic numbers. We also discuss the duality result of Hoffman, relating binomial sums and truncated multiple zeta star values.
\end{abstract}


\section{Introduction}
The multiple zeta values and their truncated counterparts are defined by
\[
\zt(i_1,\dots,i_k)=\sum_{n_1>\cdots>n_k\ge 1}\frac1{n_1^{i_1}\cdots n_k^{i_k}},
\]
with admissible indices $(i_1,\dots,i_k)$ satisfying $i_1\ge 2$, $i_j\ge 1$ for $2\le j\le k$,
and 
\[
\zt_N(i_1,\dots,i_k)=\sum_{N\ge n_1>\cdots>n_k\ge 1}\frac1{n_1^{i_1}\cdots n_k^{i_k}}.
\]
We refer to $i_1 +\dots + i_k$ as the weight of this multiple zeta value, and $k$ as its depth. 
An important variant of the (truncated) multiple zeta values are the star values, where equality is allowed:
\[
\begin{split}
  \zts(i_1,\dots,i_k) & = \sum_{n_1\ge \cdots\ge n_k\ge 1}\frac1{n_1^{i_1}\cdots n_k^{i_k}} \quad \text{and}\\
  \zts_{N}(i_1,\dots,i_k) & =\sum_{N \ge n_1\ge \cdots\ge n_k\ge 1}\frac1{n_1^{i_1}\cdots n_k^{i_k}}.
\end{split}
\]

In this article we are concerned with the series $S=S(\ell_1,\ell_2,r_1,r_2)$ and the truncated series $S_{N} = S_{N}(\ell_{1}, \ell_{2}, r_{1}, r_{2})$, defined for non-negative integers $\ell_1,\ell_2,r_1,r_2$ by
\begin{equation}
\label{eqn:sumS}
  S = \sum_{n\ge 1}\frac{\zts_{n-1}(\{1\}_{\ell_1})\zt_{n-1}(\{1\}_{\ell_2})}{\binom{n+r_1}{r_1}n^{r_2}} \quad \text{and} \quad
	S_{N} = \sum_{1 \le n \le N}\frac{\zts_{n-1}(\{1\}_{\ell_1})\zt_{n-1}(\{1\}_{\ell_2})}{\binom{n+r_1}{r_1}n^{r_2}},
\end{equation}
respectively, where $\{1\}_m$ means 1 repeated $m$ times. Throughout this work we use the convention $\zts_{n}(\{1\}_{0})=\zt_n(\{1\}_{0})=1$, for arbitrary $n$.
For $r_1+r_2\ge 2$ the sum converges and we are interested in evaluating the sum into multiple zeta values and its variants. 
We will show that the sum $S$ is closely related to several series previously discussed in the literature. 
In particular, we relate $S$ to Stirling series, Euler sums of a certain form and also to the Arakawa-Kaneko zeta function,
complementing and generalizing earlier results. Besides, we collect and survey properties of the truncated multiple zeta values
$\zt_n(\{1\}_k)$ and $\zts_n(\{1\}_k)$, as well as of binomial sums $\sum_{n=1}^{N}\binom{N}{n}\frac{(-1)^{a_1-1}}{n^{a_1}}\zeta^{\ast}_{n}(a_2,\dots,a_r)$.

\medskip

Let $\stir{n}{k}$ denote the unsigned Stirling numbers $\stir{n}{k}$ of the first kind, also called Stirling cycle numbers.
They count the number of permutations of $n$ elements with $k$ cycles~\cite{GKP} and appear as coefficients in the expansions
\begin{equation*}
\fallfak{x}n=\sum_{k=0}^{n}(-1)^{n-k}\stir{n}k x^k=\sum_{k=0}^{n}s(n,k) x^k,
\end{equation*}
relating ordinary powers $x^n$ to the so-called falling factorials $\fallfak{x}{n}=x(x-1)\dots(x-(n-1))$, for integers $n \ge 1$, and $\fallfak{x}{0}=1$. The definition can be extended to negative integers via $\fallfak{x}{-n} = \frac{1}{\fallfak{(x+n)}{n}}$, $n \ge 1$.
Here $s(n,k)$ denote the signed Stirling numbers. Stirling series of the form
\begin{equation}
\label{DefStir}
\sum_{n\ge 1}\frac{\stir{n}{d}}{n!\binom{n+s}{s}n^{r}}
\end{equation}
were studied, amongst many others, by Adamchik~\cite{A}, by the author and Prodinger~\cite{PK2010} and recently by Lyu and Wang~\cite{LW2018}. 

\smallskip

Choi~\cite{Ch} studied Euler sums of a certain form. Hoffman~\cite{H2017} generalized the results of~\cite{Ch} by 
considering two sequences of multivariate polynomials: the polynomials $P_k := P_{k}(x_{1}, \dots, x_{k})$ start with
\[
P_1(x_1) = x_1,\quad P_2(x_1, x_2) =\frac12 (x_1^2-x_2),\quad
P_3(x_1, x_2, x_3) =\frac16 (x_1^3-3x_1x_2+2x_3),\dots
\]
In fact, it holds
\[
P_k(p_1, p_2, \dots , p_k) = e_k,
\]
where, for arbitrary $n$, $p_{i} := p_i(x_{1},\dots,x_{n}) = \sum_{j=1}^{n} x_{j}^{i}$ is the $i$th power sum and $e_k := e_{k}(x_{1},\dots,x_{n}) = [x^{n-k}]\prod_{j=1}^{n}(x+x_{j})$ is the $k$th elementary symmetric function.

The polynomials $Q_k := Q_{k}(x_{1},\dots,x_{k})$ are simply obtained from the $P_k$ by skipping the signs of the coefficients, thus expressing the complete symmetric functions in terms of power sums.
Amongst others, variants of the following type of sums have been discussed in~\cite{H2017} for various pairs of $(r,s)$:
\begin{equation}
\label{DefHoffmanSum}
\sum_{n\ge 1}\frac{Q_\ell(H_n,\Ha{n}{2},\dots,\Ha{n}{\ell})P_k(H_n,\Ha{n}{2},\dots,\Ha{n}{k})}{\binom{n+s}{s} n^{r}}.
\end{equation}
Here and throughout this work we denote with $\Ha{n}{s}=\sum_{k=1}^n \frac{1}{k^s}=\zeta_n(s)$ the $n$th generalized harmonic number of order $s$ and with $H_n=\Ha{n}{1}=\zeta_n(1)$ the $n$th ordinary harmonic number. In particular, Hoffman studied the cases $(2,0)$ for arbitrary $\ell,k$ and $(0,s)$ for $\ell=0$, reducing the sums to ordinary zeta values. 

\smallskip

Furthermore, let $\Li_{i_1,\dots,i_k}(z)$ denote the multiple polylogarithm function 
defined by
\begin{equation*}
\Li_{i_1,\dots,i_k}(z)=\sum_{n_1>n_{2}>\dots>n_k\ge 1}\frac{z^{n_1}}{n_1^{i_1}n_2^{i_2}\dots n_{k}^{i_k}},
\end{equation*}
with $i_1\in\N\setminus\{1\}$ and $i_j\in\N=\{1,2,\dots\}$, $2\le j\le k$, and $|z|\le 1$.
Arakawa and Kaneko~\cite{AK1999} introduced and studied the functions $\xi_{r}(s)$ and $\xi_{i_1,\dots,i_k}(s)$,
defined by
\begin{equation*}
\begin{split}
\xi_{r}(s)&=\frac{1}{\Gamma(s)}\int_{0}^{\infty}\frac{t^{s-1}}{e^t-1}\Li_{r}(1-e^{-t})dt,\\
\xi_{i_1,\dots,i_k}(s)&=\frac{1}{\Gamma(s)}\int_{0}^{\infty}\frac{t^{s-1}}{e^t-1}\Li_{i_1,\dots,i_k}(1-e^{-t})dt,
\end{split}
\end{equation*}
respectively, being absolutely convergent for $\Re(s)>0$. Arakawa and Kaneko related $\xi_{r}(s)$ and $\xi_{i_1,\dots,i_k}(s)$ for several choices of $s\in\C$, and $r \in \N$ and $i_1,\dots i_k\in\N$, respectively, to multiple zeta values. 

\medskip

\section{Main results}
Our main results are evaluations of $S$ as defined in \eqref{eqn:sumS} and its truncated counterparts, as well as their relations to other sums. Our first observation is that sums of Hoffman's type~\eqref{DefHoffmanSum} and the sum $S$ are almost identical, except for a shift in the arguments of the truncated zeta functions.
\begin{thm}
\label{thmHoffman}
The series $S=\sum_{n\ge 1}\frac{\zts_{n-1}(\{1\}_{\ell_1})\zt_{n-1}(\{1\}_{\ell_2})}{\binom{n+r_1}{r_1}n^{r_2}}$ satisfies:
\[
S=\sum_{n\ge 1}\frac{Q_{\ell_1}(H_{n-1},\Ha{n-1}{2},\dots,\Ha{n-1}{\ell_1})P_{\ell_2}(H_{n-1},\Ha{n-1}{2},\dots,\Ha{n-1}{\ell_2})}{\binom{n+r_1}{r_1} n^{r_2}}.
\]
\end{thm}

Next, we evaluate and relate $S$ to other sums.

\subsection{Stirling series}
 For $\ell_1=0$ we obtain Stirling series, already evaluated by Prodinger and the author in~\cite{PK2010}. We collect their result and extend it by giving a truncated counterpart, generalizing a result of Spie\ss{}~\cite{S}.

	\begin{thm}
	\label{thmStirlingFinite}
	The truncated series 
	\begin{equation*}
	  S_N = S_{N}(0,\ell_{2},r_{1},r_{2}) = \sum_{n=1}^{N}\frac{\zt_{n-1}(\{1\}_{\ell_2})}{\binom{n+r_1}{r_1}n^{r_2}}
	\end{equation*}
	is for $r_{1}, r_{2} \ge 1$ and $N \ge \ell_{2} \ge 0$ given by the following expression:
	\begin{equation*}
    S_N = \sum_{m=2}^{r_2}(-1)^{r_2-m}\zt_{N}(m,\{1\}_{\ell_{2}})\zts_{r_1}(\{1\}_{r_2-m}) + K(r_{1},r_{2},\ell_{2}) 
		+ E_{N}(r_1,r_2,\ell_{2}),
	\end{equation*}
	with
	\begin{equation*}
	  K(r_{1},r_{2},\ell_{2}) = \begin{cases}
		  (-1)^{r_{2}+1} \zts_{r_{1}}(\{1\}_{r_{2}-1},\ell_{2}+2), & \quad r_{2} \ge 2, \\
			\frac{1}{r_{1}^{\ell_{2}+1}}, & \quad r_{2}=1,
		\end{cases}
	\end{equation*}
	and
	\begin{equation*}
	  E_{N}(r_1,r_2,\ell_{2}) = (-1)^{r_2+1}\sum_{k_{0}=1}^{r_1}\binom{r_1}{k_{0}}\frac{(-1)^{k_{0}+1}}{k_{0}^{r_2-1}}R_N(\ell_{2},k_{0}),
	\end{equation*}
	where $R_{N}(\ell_{2},k_{0})$ is given by the nested sum
	\begin{equation*}
	  R_{N}(\ell_{2},k_{0}) = \sum_{j=0}^{\ell_{2}} \left[\sum_{k_{1}=1}^{k_{0}} \frac{1}{k_{1}} \sum_{k_{2}=1}^{k_{1}} \frac{1}{k_{2}} \cdots \sum_{k_{j}=1}^{k_{j-1}} \frac{1}{k_{j}} \cdot \zt_{N-1-j}(\{1\}_{\ell_{2}-j}) (H_{N-j}-H_{N-j+k_{j}})\right].
	\end{equation*}
	Moreover, $\lim_{N\to\infty} E_{N}(r_1,r_2,\ell_{2}) = 0$, such that $\lim_{N\to\infty} S_N = S$.
	\end{thm}

By taking the limit $N\to\infty$ of the truncated series we obtain the following corollary.

\begin{coroll}
\label{thmStirling}
The series $S=S(0,\ell_2,r_1,r_2)$ is given by a Stirling series,
	\[
	S = \sum_{n\ge 1}\frac{\zt_{n-1}(\{1\}_{\ell_2})}{\binom{n+r_1}{r_1}n^{r_2}}
	= \sum_{n\ge 1}\frac{\stir{n}{\ell_2+1}}{n!\binom{n+r_1}{r_1}n^{r_2-1}},
	\]
	and it satisfies for $r_{1}, r_{2} \ge 1$:
	\begin{equation*}
	  S = \begin{cases}
		      \sum_{m=2}^{r_2}(-1)^{r_2-m}\zt(m,\{1\}_{\ell_2})\zts_{r_1}(\{1\}_{r_2-m})\\
			     \quad \mbox{} +	(-1)^{r_2+1}\zts_{r_1}(\{1\}_{r_2-2},\ell_{2}+2), & \quad \text{for $r_{2} \ge 2$},\\
	    \frac{1}{r_{1}^{\ell_{2}+1}}, & \quad \text{for $r_{2}=1$}. 
	  \end{cases}
	\end{equation*}
\end{coroll}	
Moreover, a direct byproduct is the following observation.
\begin{coroll}
\label{CoroStirling}
The series $S=S(0,\ell_2,r_1,r_2)$ is for arbitrary $r_{1}, r_{2} \ge 1$ a rational polynomial in the functions $\zt(i)$.
\end{coroll}

\subsection{Zeta star series}

	\begin{thm}
	\label{thmZetastar}
	The series $S = S(\ell_{1}, 0, r_{1}, r_{2}) = \sum_{n\ge 1}\frac{\zts_{n-1}(\{1\}_{\ell_1})}{\binom{n+r_1}{r_1}n^{r_2}}$ can be evaluated as follows:
	\begin{align*}
	S&=	\sum_{m=2}^{r_2}(-1)^{r_2-m}\big(\zts(m,\{1\}_{\ell_1})-\zts(m+1,\{1\}_{\ell_1-1})\big)\zts_{r_1}(\{1\}_{r_2-m})\\
	&\quad+(-1)^{r_2+1}\sum_{k=1}^{r_1}\binom{r_1}{k}\frac{(-1)^{k+1}}{k^{r_2-1}}\Big[T^{\ast}(\ell_1,k)+\frac1kT^{\ast}(\ell_1-1,k)
		\Big]\\
		&\quad+(-1)^{r_2}\zts(2,\{1\}_{\ell_1-1})\zts_{r_1}(\{1\}_{r_2-1}),
	\end{align*}
	with
	\begin{equation*}
	  T^{\ast}(\ell,k) = 
		\begin{cases}
		  \sum_{j=0}^{\ell-1}\zts(2,\{1\}_{\ell-1-j})\zt_{k-1}(\{1\}_j)\\
			\quad \mbox{} +\zt_{k-1}(\{1\}_{\ell+1})+\zt_{k-1}(\{1\}_{\ell-1},2), & \quad \text{for $\ell \ge 1$},\\
			H_{k}, & \quad \text{for $\ell=0$}.
		\end{cases}
	\end{equation*}
\end{thm}
\begin{remark}\label{rem:zeta_star}
The alternating binomial sums involving $\zt_{k-1}(a_1,\dots,a_j)$ occurring in the representation of $S$ given in Theorem~\ref{thmZetastar} can be converted to truncated multiple zeta star values. We comment on this in Section~\ref{sec:zeta_star}. 
\end{remark}
\begin{remark}
We also have a corresponding result for the truncated series $S_{N} = \sum_{n=1}^{N}\frac{\zts_{n-1}(\{1\}_{\ell_1})}{\binom{n+r_1}{r_1}n^{r_2}}$, 
which is more involved compared to the truncated Stirling series in Theorem~\ref{thmStirlingFinite}.
\end{remark}

\begin{coroll}
\label{corollZTS}
The series 
\[
S = S(\ell_1,0,r_1,1)=\sum_{n\ge 1}\frac{\zts_{n-1}(\{1\}_{\ell_{1}}}{\binom{n+r_1}{r_1}n}
\]
is for $r_{1} \ge 1$ a rational polynomial in the $\zt(i)$. 

\smallskip

Similarly, the series $S = S(\ell_{1},0,r_{1},2)$ is for $\ell_{1} \le 2$ a rational polynomial in the $\zt(i)$. 
\end{coroll}

\subsection{Relation to the Arakawa-Kaneko zeta function}

\begin{thm}
\label{thmArakawa}
For $r_1=0$ the series $S$ can be expressed in terms of the Arakawa-Kaneko zeta function:
	\[
	S = S(\ell_{1},\ell_{2},0,r_{2}) = \sum_{n\ge 1}\frac{\zts_{n-1}(\{1\}_{\ell_1})\zt_{n-1}(\{1\}_{\ell_2})}{n^{r_2}}
	=\xi_{r_2-1,\{1\}_{\ell_2}}(\ell_1+1) - \xi_{r_2,\{1\}_{\ell_2}}(\ell_1).
	\]
	
On the other hand, special values of the Arakawa-Kaneko zeta function can be written
in terms of the polynomials $P_k$ and $Q_k$. Namely, for $s\in\N$, it holds
\[
\xi_{i_{1},\{1\}_{r-1}}(s)=
\sum_{n \ge 1}\frac{Q_{s-1}(H_{n},\Ha{n}{2},\dots,\Ha{n}{s-1})P_{r-1}(H_{n-1},\Ha{n-1}{2},\dots,\Ha{n-1}{r-1})}{n^{i_{1}+1}}.
\]
\end{thm}
\begin{remark}
The result above allows to translate identities for the Arakawa-Kaneko zeta function 
into evaluations of $S$. On the other hand, the presentation of $\xi_{i_{1},\{1\}_{r-1}}(s)$
allows to translate results of Hoffman~\cite{H2017} into evaluations of $\xi_{i_{1},\{1\}_{r-1}}(s)$.
Values of the ordinary Arakawa-Kaneko zeta function, case $r=1$, only involve $Q_{s-1}(H_{n},\Ha{n}{2},\dots,\Ha{n}{s-1})$, but not the polynomials $P_k$.
\end{remark}

The general case is more involved. We can give an evaluation into the Arakawa-Kaneko zeta function
and another sum. This is discussed in the final section along with open problems.

\medskip

\section{Truncated multiple zeta values}
Before we turn to a proof of the main theorems we collect in this section several results for the truncated zeta values $\zt_n(\{1\}_k)$ and $\zts_n(\{1\}_k)$. 
Following Hoffman~\cite{H2017}, we note two different explicit expressions for the polynomials $P_k$ and $Q_k$ due to MacDonald~\cite{MacDo}.
\begin{lem}
\label{lemHoffman}
The polynomials $P_k$ and $Q_k$ satisfy
\[
P_k(x_1,\dots,x_k)=\sum_{m_1+2m_2+\dots =k}\frac{(-1)^{m_2+m_4+\dots}}{m_1!m_2!\dots}
\Big(\frac{x_1}{1}\Big)^{m_1}\Big(\frac{x_2}{2}\Big)^{m_2}\dots,
\]
\[
Q_k(x_1,\dots,x_k)=\sum_{m_1+2m_2+\dots =k}\frac{1}{m_1!m_2!\dots}
\Big(\frac{x_1}{1}\Big)^{m_1}\Big(\frac{x_2}{2}\Big)^{m_2}\dots .\\
\]
\end{lem}
First, we discuss expressions for $\zt_{n}(\{1\}_{k})$. These expressions are, for example in terms of the Stirling numbers of the first kind, well-known (see for example Adamchik~\cite{A}, Prodinger~\cite{PK2010}, Hoffman~\cite{H2017}, Hoffman et al.~\cite{H2018+}), but perhaps not all parts of it and not in this notation.
\begin{lem}[Truncated multiple zeta values $\zt_n(\{1\}_k)$]
\label{stirmzv}
For positive integers $n\ge k$, the truncated multiple zeta values can be expressed as (weighted) Stirling numbers of the first kind,
\[
\zt_{n}(\{1\}_{k})=(-1)^{n-k}\frac{s(n+1,k+1)}{n!}= \frac{\stir{n+1}{k+1}}{n!},
\]
in terms of Bell polynomials and generalized Harmonic numbers, 
\begin{align*}
\zt_{n}(\{1\}_{k})&=\frac{(-1)^{k}}{k!}B_{k}(-0!\Ha{n}{1},-1!\Ha{n}{2},\dots,-(k-1)!\Ha{n}{k})\\
&=\sum_{m_1+2m_2+\dots=k}\frac{(-1)^{m_2+m_4+\dots}}{m_1!m_2!\dots}\Big(\frac{\Ha{n}{1}}{1}\Big)^{m_1} \Big(\frac{\Ha{n}{2}}{2}\Big)^{m_2}\dots,
\end{align*}
as well as by the polynomials $P_k$ and via determinants:
\[
\zt_{n}(\{1\}_{k})=P_k(\Ha{n}{1},\dots,\Ha{n}{k})=\frac{1}{k!}
\left|
\begin{matrix}
\Ha{n}{1} & 1 & 0 & \dots &0\\
\Ha{n}2 & \Ha{n}{1} & 2 & \dots &0\\
\hdots & \hdots & \ddots & \ddots &\vdots\\
\Ha{n}{k-1} & \Ha{n}{k-2} & \Ha{n}{k-3} & \ddots &k-1\\
\Ha{n}{k} & \Ha{n}{k-1} & \Ha{n}{k-2} & \dots &\Ha{n}1\\
\end{matrix}
\right|.
\]
\end{lem}

\begin{proof}
From the relation
\[
x(x-1)\cdots (x-n+1)=\sum_{k=0}^n s(n,k)x^k
\]
it follows that $s(n,k)=(-1)^{n-k}e_{n-k}(1,2,\dots,n-1)$, with $e_j$
the $j$th elementary symmetric function. After dividing by $(n-1)!$ we get
\[
\frac{s(n,k)}{(n-1)!}=(-1)^{n-k}\frac{e_{n-k}(1,2,\dots,n-1)}{(n-1)!}
=(-1)^{n-k}e_{k-1}\left(1,\frac12,\dots,\frac1{n-1}\right),
\]
from which the first conclusion follows, since evidently $\zt_{n-1}(\{1\}_{k})=
e_{k}\left(1,\frac12,\dots,\frac1{n-1}\right)$.
Alternatively, the relation to Stirling numbers of first kind can be shown by using generating series. Namely, each summand of $\zt_{n}(\{1\}_{k})$ corresponds to a $k$-elementary set of $[n]=\{1, \dots, n\}$, where each element $i$ gets multiplicative weight $\frac{1}{i}$. This shows the relation
\begin{equation*}
  \zt_{n}(\{1\}_{k}) = [q^{k}] \big(1+\frac{q}{1}\big) \cdot \big(1+\frac{q}{2}\big) \cdot \cdots \cdot \big(1+\frac{q}{n}\big).
\end{equation*}
The bivariate generating series of $\zt_{n}(\{1\}_{k})$ is thus given as follows:
\begin{equation*}
  \sum_{n,k \ge 0} \zt_{n}(\{1\}_{k}) \, z^{n} q^{k} = \sum_{n \ge 0} \binom{n+q}{n} z^{n} = \frac{1}{(1-z)^{q+1}}.
\end{equation*}
Using the well-known bivariate generating series of the unsigned Stirling numbers of first kind~\cite{GKP},
$\sum_{n,k \ge 0} \stir{n}{k} \frac{z^{n}}{n!} q^{k} = \frac{1}{(1-z)^{q}}$, taking the derivative w.r.t.\ $q$ and extracting coefficients, immediately leads to the stated result.

Moreover, from the $\exp-\log$-representation we get
\begin{align*}
z(z-1)\cdots (z-n+1)&=(-1)^{n-1}(n-1)!z\exp(\sum_{k=1}^{n-1}\log(1-\frac{z}{k}))\\
&=(-1)^{n-1}(n-1)!z\exp(\sum_{j\ge 1}(-\Ha{n-1}{j})\frac{z^j}{j}).
\end{align*}
Thus, it follows that $s(n,k)$ can be expressed in terms of the complete Bell polynomials $B_n(x_1,\dots,x_n)$, which are defined via
\[
\exp\Big(\sum_{\ell\ge 1}\frac{x_\ell}{\ell!}z^{\ell}\Big)
= \sum_{j\ge 0}\frac{B_j(x_1,\dots,x_j)}{j!}z^j,
\]
evaluated at $x_\ell=-(\ell-1)!\Ha{n}{\ell}$. Furthermore, according to the definition of the Bell polynomials we obtain the stated expression in terms of the generalized harmonic numbers. In particular, they coincide with the expression for $P_k$
as given in Lemma~\ref{lemHoffman} and thus we get the stated determinant form from the formula for the polynomials $P_k$ by MacDonald~\cite{MacDo}.
Another way, avoiding the Stirling numbers of the first kind, is to use the algebraic machinery established
by Hoffman and Ihara~\cite{HI} to obtain directly the Bell polynomial expression.
\end{proof}

Next we discuss expressions for the truncated multiple zeta star values $\zts_n(\{1\}_k)$. 
The first one is usually attributed to Dilcher~\cite{D}, although it had occurred earlier in the literature. The second one goes back to Flajolet and Sedgewick~\cite{FS}. Both expressions have been rediscovered a few times. For related results, a comprehensive historical discussion, as well as many different proofs, see, e.g., the work of Batir~\cite{B}. Moreover, parts of the results stated below, as well as a few other expressions, can also be found in the work of Bai et al.~\cite{Bai}.
\begin{lem}[Truncated zeta star values $\zts_n(\{1\}_k)$]
\label{ztsTrunc}
For positive integers $n \ge k$, the values $\zts_{n}(\{1\}_{k})$ can be expressed as 
\[
\zts_{n}(\{1\}_{k})=\sum_{j=1}^{n}\binom{n}j\frac{(-1)^{j-1}}{j^k},
\]
in terms of Bell polynomials and generalized Harmonic numbers,
\begin{align*}
\zts_{n}(\{1\}_{k})&=\frac{1}{(k-1)!}(B_{k-1}(0!\Ha{n}{1},1!\Ha{n}{2},\dots,(k-1)!\Ha{n}{k})\\
&=\sum_{m_1+2m_2+\dots=k}\frac{1}{m_1!m_2!\dots}\Big(\frac{\Ha{n}{1}}{1}\Big)^{m_1} \Big(\frac{\Ha{n}{2}}{2}\Big)^{m_2}\dots,
\end{align*}
as well as by the polynomials $Q_k$ and via determinants:
\[
\zts_{n}(\{1\}_{k})=Q_k(H_n^{(1)},\dots,H_n^{(k)}) =\frac{1}{k!}
\left|
\begin{matrix}
\Ha{n}{1} & -1 & 0 & \dots &0\\
\Ha{n}2 & \Ha{n}{1} & -2 & \dots &0\\
\hdots & \hdots & \hdots & \vdots &\hdots\\
\Ha{n}{k-1} & \Ha{n}{k-2} & \Ha{n}{k-3} & \dots &-(k-1)\\
\Ha{n}{k} & \Ha{n}{k-1} & \Ha{n}{k-2} & \dots &\Ha{n}1\\
\end{matrix}
\right|.
\]
Moreover, they satisfy the recurrence relation
\[
\zeta_{n}^{\ast}(\{1\}_{k})=\zeta_{n-1}^{\ast}(\{1\}_{k}) + \frac1n \zeta_{n}^{\ast}(\{1\}_{k-1}),\quad k\ge 1,
\]
with $\zeta_{n}^{\ast}(\{1\}_{0})=1$, such that
\[
\zeta_{n}^{\ast}(\{1\}_{k})=\sum_{m=0}^{k} \frac1{n^{m}}\zeta_{n-1}^{\ast}(\{1\}_{k-m}).
\]
\end{lem}

\begin{remark}
\label{remPWschlkt}
The truncated zeta star values $\zts_{n}(\{1\}_{k})$ are also of importance in different fields of applications.
For example, let $X_j=\Exp(1)$, $1\le j\le n$, denote mutually independent standard exponentially distributed random variables.
The maximum $Z_{n} = \max\{X_1,\dots,X_n\}$ has density $f_{Z_{n}}(x)=ne^{-x}(1-e^{-x})^{n-1}$, for $x \ge 0$.
Since
\begin{align*}
  & \int_{0}^{\infty} x^{k} \, f_{Z_{n}}(x) dx = n \sum_{j=0}^{n-1} \binom{n-1}{j} (-1)^{j} \int_{0}^{\infty} x^{k} e^{-(1+j)x} dx\\
	& \quad = n \sum_{j=0}^{n-1} \binom{n-1}{j} (-1)^{j} \frac{1}{(1+j)^{k+1}} \int_{0}^{\infty} t^{k} e^{-t} dt 
	= k! \sum_{j=1}^{n} \binom{n}{j} (-1)^{j-1} \frac{1}{j^{k}},
\end{align*}
an application of Lemma~\ref{ztsTrunc} implies that the positive integer moments of $Z_{n}$ are given in terms of the truncated multiple zeta star values:
\begin{equation*}
  \E(Z_n^k) = n \int_{0}^{\infty} x^{k} e^{-x} (1-e^{-x})^{n-1} dx = k! \, \zts_n(\{1\}_k).
\end{equation*}
We note in passing that the limit $n\to\infty$ of $Z_n$, suitably centered and normalized, 
tends to a Gumbel distributed random variable $G$, whose cumulants $\kappa_s(G)$ are also related to zeta values: 
$\kappa_s(G)=(s-1)! \, \zeta(s)$, $s\ge 2$. 

\smallskip

Moreover, $\zts_{n}(\{1\}_{k})$ also occurs in connection with the number of maxima in hypercubes, see~\cite{Bai}
for several additional different expressions; for example, let $P=1-U_1U_2\dots U_{k+1}$, where the $U_i$, $1 \le i \le k+1$, denote independent standard uniformly distributed random variables. Then, again by an application of Lemma~\ref{ztsTrunc}, one gets
$$\E(P^{n-1})=\frac1n\zts_{n}(\{1\}_{k})
.$$
\end{remark}

For the sake of completeness we collect short proofs of Lemma~\ref{ztsTrunc}.
\begin{proof}
The first part follows by induction and repeated usage of the formula
$\sum_{\ell=k}^{n}\binom{\ell-1}{k-1}=\binom{n}{k}$, as well as 
$\sum_{\ell=1}^{k}(-1)^{\ell-1}\binom{k}{\ell}=1$.
Following~\cite{FS}, the second part is based on complex analysis and a Rice integral representation. 
The argument can be reduced to a coefficient extraction from the generating series of $\zts_n(\{1\}_k)$:
\[
\zts_n(\{1\}_k)=[q^{k}]\frac{1}{(1-\frac{q}{1})(1-\frac{q}{2})\dots (1-\frac{q}{n})}.
\]
We remark that the generating series expression also follows from the observation that each summand of $\zts_{n}(\{1\}_{k})$ corresponds to a $k$-elementary multiset of $[n]$, where each element $i$ gets multiplicative weight $\frac{1}{i}$. 
Proceeding as in the proof of Lemma~\ref{stirmzv} using the $\exp-\log$-representation leads to the stated expression in terms of the Bell polynomials.
Alternatively, the algebraic machinery of~\cite{HI} also leads to the Bell polynomial expression, avoiding complex analysis.
It coincides with the expression for $Q_k$ as given in Lemma~\ref{lemHoffman}. The determinant form follows readily from the formula for the polynomials $Q_k$ by MacDonald~\cite{MacDo}.
Finally, the stated recurrence relation immediately follows from the definition:
\begin{align*}
\sum_{N\ge n_1\ge \cdots\ge n_k\ge 1}\frac1{n_1^{i_1}\cdots n_k^{i_k}}
&=\sum_{N-1\ge n_1\ge \cdots\ge n_k\ge 1}\frac1{n_1^{i_1}\cdots n_k^{i_k}}\\
&\quad+ \sum_{N=n_1\ge n_2\ge  \cdots\ge n_k\ge 1}\frac1{n_1^{i_1}\cdots n_k^{i_k}}.
\end{align*}
It is readily solved and leads directly to the stated result.

\end{proof}

\section{Proof of the main results\label{sec:ProofResults}}
%
%

\begin{proof}[Proof of Theorem~\ref{thmHoffman}]
The equivalence between Hoffman's sum and $S$ follows readily by Lemmata~\ref{stirmzv} and~\ref{ztsTrunc}, 
expressing $\zt_{n-1}(\{1\}_k)$ and $\zts_{n-1}(\{1\}_k)$ in terms of the polynomials $P_k$ and $Q_k$.
\end{proof}

\subsection{Proofs for the Stirling series}
We collect a basic result (see Graham, Knuth, Patashnik~\cite{GKP}) concerning finite differences.
\begin{lem}{(Finite differences)}
\label{LemmaKnuth}
For $x\notin\{0,-1,\dots,-n\}$ it holds
\[
\frac{1}{x\binom{x+m}{m}}=\sum_{k=0}^{m}\binom{m}{k}\frac{(-1)^k}{x+k}.
\]
\end{lem}
\begin{proof}
Let $\Delta=E-I$ be the forward difference operator, with $E$ denoting the shift operator and $I$ the identity, such that
$\Delta f(x)= f(x+1)-f(x)$. By the binomial theorem 
\[
\Delta^m=(E-I)^m=\sum_{k=0}^{m}\binom{m}{k}E^k(-1)^{m-k}.
\]
$\Delta$ acts as a derivative operator on falling factorials, i.e., $\Delta x^{\underline{n}} = n x^{\underline{n-1}}$, for $n \in \mathbb{Z}$. Hence, the application of $\Delta^m$ to $\frac{1}{x}$ gives $\frac{m!(-1)^m }{x(x+1)\dots (x+m)}$. On the other hand, $E^k\frac{1}{x}=\frac{1}{x+k}$ and the result follows.
\end{proof}

Next we collect a result, which generalizes a result of Batir~\cite{B}. We note in passing that a much more general result is due to Hoffman. It is collected later on in Lemma~\ref{lemBr}, together with a simple proof.
\begin{lem}[Nested truncated zeta star values]
\label{ztsNested}
The truncated zeta star values satisfy, for integers $n \ge 1$ and $k, \ell \ge 0$:
\[
\sum_{j=1}^{n}\binom{n}j\frac{(-1)^{j-1}}{j^k}\zts_j(\{1\}_{\ell}) =
\begin{cases}
  \zts_{n}(\{1\}_{k-1},\ell+1), & \quad \text{for $k \ge 1$},\\
	\frac{1}{n^{\ell}}, & \quad \text{for $k=0$}.
\end{cases}
\]
\end{lem}
\begin{proof}
We use the fact that $\sum_{i=j}^{n}\binom{i-1}{j-1}=\binom{n}{j}$. 
Thus, for any sequence $(a_j)_{j\in\N}$ and $k\ge 1$, we have
\begin{align*}
&\sum_{j=1}^{n}\binom{n}j (-1)^{j-1} \frac{a_j}{j^k}= \sum_{j=1}^{n}\sum_{i=j}^{n}\binom{i-1}{j-1}(-1)^{j-1}\frac{a_j}{j^k}
=\sum_{i=1}^{n}\frac1i \sum_{j=1}^{i}\binom{i}{j}(-1)^{j-1}\frac{a_j}{j^{k-1}}.
\end{align*}
Iterating this argument implies an often rediscovered nested sum expression:
\begin{equation}
\label{BradleySimple}
\sum_{j=1}^{n}\binom{n}j(-1)^{j-1} \frac{a_j}{j^k} = \sum_{i_1=1}^{n}\frac1{i_1} \sum_{i_{2}=1}^{i_{1}} \frac{1}{i_{2}} \dots \sum_{i_{k}=1}^{i_{k-1}}\frac1{i_{k}}\sum_{j=1}^{i_{k}}\binom{i_{k}}{j}(-1)^{j-1}a_j.
\end{equation}

Starting from the expression above we set $a_j=\zts_j(\{1\}_{\ell})$. Here, we can give two different arguments yielding Lemma~\ref{ztsNested}.

First, we use the inversion formula (see for example~\cite{GKP}):
\[
a_n=\sum_{j=1}^n\binom{n}{j}(-1)^j b_j \quad\Leftrightarrow \quad
b_n=\sum_{j=1}^n\binom{n}{j}(-1)^j a_j.
\]
By Lemma~\ref{ztsTrunc} this implies that 
\[
\frac{1}{n^{\ell}}=\sum_{j=1}^n\binom{n}{j}(-1)^{j-1} \zts_j(\{1\}_{\ell}),
\]
which, in passing, shows the case $k=0$, and further
\[
\sum_{j=1}^{n}\binom{n}j\frac{(-1)^{j-1}}{j^k}\zts_j(\{1\}_{\ell})
=\sum_{i_1=1}^{n}\frac1{i_1}\dots \sum_{i_{k}=1}^{i_{k-1}}\frac1{i_{k}}\cdot \frac{1}{i_k^{\ell}}
=\zts_n(\{1\}_{k-1},\ell+1),
\]
which proves the stated result for $k \ge 1$. 

\smallskip 

Second, we change the order of summation:
\[
\sum_{j=1}^{i_{k}}\binom{i_{k}}{j}(-1)^{j-1}\zts_j(\{1\}_{\ell})=
\sum_{h=1}^{i_k}\frac{\zts_h(\{1\}_{\ell-1})}{h}\sum_{j=h}^{i_k}\binom{i_{k}}{j}(-1)^{j-1}.
\]
Writing $\binom{i_{k}}{j} (-1)^{j-1} = \binom{i_{k}-1}{j} (-1)^{j+1} - \binom{i_{k}-1}{j-1} (-1)^{j}$, the latter sum telescopes yielding
\[
\sum_{j=h}^{i_k}\binom{i_{k}}{j}(-1)^{j-1}=(-1)^{h-1}\binom{i_{k}-1}{h-1}=(-1)^{h-1}\frac{h}{i_k}\binom{i_{k}}{h}. 
\]
Thus, for $\ell \ge 1$,
\begin{equation}
\sum_{j=1}^{i_{k}}\binom{i_{k}}{j}(-1)^{j-1}\zts_j(\{1\}_{\ell}) = \frac{1}{i_k}\sum_{h=1}^{i_k} \binom{i_{k}}{h} (-1)^{h-1}\zts_h(\{1\}_{\ell-1}).
\label{Iterating}
\end{equation}
Iterating this recurrence relation and using the case $\ell=0$, $\sum_{j=1}^{i_{k}}\binom{i_{k}}{j}(-1)^{j-1}=1$, yields for $\ell \ge 0$
\[
\sum_{j=1}^{i_{k}}\binom{i_{k}}{j}(-1)^{j-1}\zts_j(\{1\}_{\ell}) = \frac{1}{i_k^{\ell}},
\]
which, together with \eqref{BradleySimple}, also proves Lemma~\ref{ztsNested}.
\end{proof}

\begin{proof}[Proof of Theorem~\ref{thmStirlingFinite}]
The relation between the truncated series $S_{N}$ and the Stirling series follows directly from Lemma~\ref{stirmzv}. In order to evaluate $S_N$ for $\ell_1=0$ 
we proceed similar to the proofs of a certain evaluation by Panholzer and Prodinger~\cite{PP}, see also~\cite{PK2010}. By Lemma~\ref{LemmaKnuth} we have
\[
\frac{1}{r_1\binom{n+r_1}{r_1}}=\frac{1}{(n+1)\binom{n+r_1}{r_1-1}}=\sum_{k=0}^{r_1-1}\binom{r_1-1}{k}\frac{(-1)^k}{n+1+k},
\]
which gives the representation
\[
S_N=r_1\sum_{k=1}^{r_1}(-1)^{k+1}\binom{r_1-1}{k-1}\sum_{n=1}^{N}\frac{\zt_{n-1}(\{1\}_{\ell_{2}})}{(n+k)n^{r_2}}.
\]
In order to proceed we use partial fraction decomposition:
\begin{equation}
\label{parfrac}
\frac1{(n+k)n^{r_2}}=\sum_{m=2}^{r_2}\frac{(-1)^{r_2-m}}{n^m k^{r_2+1-m}} + \frac{(-1)^{r_2+1}}{k^{r_2}}\Big(\frac{1}{n}-\frac{1}{n+k}\Big).
\end{equation}
This implies that the sum $S_N$ can be decomposed into two parts $S_{N;1}$ and $S_{N;2}$:
\begin{align*}
S_N&=r_1\sum_{k=1}^{r_1}(-1)^{k+1}\binom{r_1-1}{k-1}
\bigg[\sum_{m=2}^{r_2}\frac{(-1)^{r_2-m}}{ k^{r_2+1-m}} \sum_{n=1}^{N} \frac{\zt_{n-1}(\{1\}_{\ell_{2}})}{n^m}\bigg]\\
&+r_1\sum_{k=1}^{r_1}(-1)^{k+1}\binom{r_1-1}{k-1}\frac{(-1)^{r_2+1}}{k^{r_2}} \sum_{n=1}^{N}\zt_{n-1}(\{1\}_{\ell_{2}})\Big(\frac{1}{n}-\frac{1}{n+k}\Big).
\end{align*}
The first part is readily simplified:
\begin{align*}
S_{N;1}&=r_1\sum_{k=1}^{r_1}(-1)^{k+1}\binom{r_1-1}{k-1}\sum_{m=2}^{r_2}\frac{(-1)^{r_2-m}}{ k^{r_2+1-m}}
\zt_{N}(m,\{1\}_{\ell_{2}})\\
&=\sum_{m=2}^{r_2}(-1)^{r_2-m}\zt_{N}(m,\{1\}_{\ell_{2}}) \sum_{k=1}^{r_1}(-1)^{k+1}\binom{r_1}{k}\frac1{k^{r_2-m}}\\
&=\sum_{m=2}^{r_2}(-1)^{r_2-m}\zt_{N}(m,\{1\}_{\ell_{2}})\zts_{r_1}(\{1\}_{r_2-m}).
\end{align*}
For the second part we introduce a new notation: let 
\begin{equation}
T_N(\ell,k):= \sum_{n=1}^{N}\zt_{n-1}(\{1\}_{\ell})\Big(\frac{1}{n}-\frac{1}{n+k}\Big).
\label{truncatedT}
\end{equation}
For $\ell=0$ we have 
\begin{align*}
T_N(0,k)&=\sum_{n=1}^{N}\zt_{n-1}(\{1\}_{0})\Big(\frac{1}{n}-\frac{1}{n+k}\Big)
=\sum_{n=1}^{N}\Big(\frac{1}{n}-\frac{1}{n+k}\Big)
=H_N - H_{N+k} +H_k.
\end{align*}

For $\ell\ge 1$ we get a recurrence relation:
\begin{align*}
T_N(\ell,k)&=\sum_{n=1}^{N}\zt_{n-1}(\{1\}_\ell)\Big(\frac{1}{n}-\frac{1}{n+k}\Big)
=\sum_{j=1}^{N-1}\frac{\zt_{j-1}(\{1\}_{\ell-1})}j \sum_{n=j+1}^N \Big(\frac{1}{n}-\frac{1}{n+k}\Big)\\
&=\sum_{j=1}^{N-1}\frac{\zt_{j-1}(\{1\}_{\ell-1})}j\Big(H_N-H_j - H_{N+k}+H_{j+k}\Big)\\
&=\sum_{j=1}^{N-1}\frac{\zt_{j-1}(\{1\}_{\ell-1})}j\Big(\sum_{i=1}^{k}\frac1{j+i} +H_N - H_{N+k}\Big)\\
&=\zt_{N-1}(\{1\}_{\ell})(H_N - H_{N+k}) +\sum_{i=1}^{k}\frac1i \sum_{j=1}^{N-1}\zt_{j-1}(\{1\}_{\ell-1})\Big(\frac1j-\frac1{j+i}\Big)\\
&=\zt_{N-1}(\{1\}_{\ell})(H_N - H_{N+k}) +\sum_{i=1}^{k}\frac1i T_{N-1}(\ell-1,i).
\end{align*}
In the following let $k_0=k$. Assuming $N \ge \ell$ gives, by unwinding the recurrence relation, a nested sum:
\begin{align*}
T_N(\ell,k)&=R_N(\ell,k)+ \zts_k(\{1\}_{\ell+1}),
\end{align*}
with
\begin{equation}
R_N(\ell,k)=\sum_{j=0}^{\ell}\bigg[ \sum_{k_1=1}^{k_{0}}\frac1{k_1}\sum_{k_2=1}^{k_1}\frac1{k_2}\dots
\sum_{k_{j}=1}^{k_{j-1}} \frac{1}{k_{j}} \zt_{N-1-j}(\{1\}_{\ell-j}) (H_{N-j} - H_{N-j+k_{j}})\bigg].
\label{ErrorTermR}
\end{equation}
Collecting our intermediate results we get
\begin{align*}
S_{N;2}&=\sum_{k=1}^{r_1}(-1)^{k+1}\binom{r_1}{k}\frac{(-1)^{r_2+1}}{k^{r_2-1}}T_N(\ell_{2},k)\\
&=\sum_{k=1}^{r_1}(-1)^{k+1}\binom{r_1}{k}\frac{(-1)^{r_2+1}}{k^{r_2-1}}(\zts_k(\{1\}_{\ell_{2}+1})+R_N(\ell_{2},k)).
\end{align*}
Applying Lemma~\ref{ztsNested} to $S_{N;2}$ and adding $S_{N;1}$ directly leads to the result stated in Theorem~\ref{thmStirling}.

In order to show the asymptotic behaviour we require the asymptotic expansion of the harmonic numbers, for $n \to \infty$,
\[
H_{n}=\log n +\gamma
        +\frac{1}{2n}-\frac{1}{12n^{2}}+\mathcal{O}\Bigl(\frac{1}{n^{4}}\Bigr).
\]
First we obtain $H_{N-j} - H_{N-j+k_{j}}=\mathcal{O}(1/N)$. Moreover, $\zt_{N}(\{1\}_{\ell}) = \mathcal{O}\big(H_{N}^{\ell}\big) = \mathcal{O}(\ln^{\ell} N)$, which implies $R_N(\ell,k))=\mathcal{O}\Big(\frac{\ln^{\ell} N}{N}\big)$.
\end{proof}

\begin{proof}[Proof of Corollary~\ref{CoroStirling}]
For all positive integers $n,m$ the multiple zeta value $\zt(m+1,\{1\}_n)$ is 
a rational polynomial in the $\zt(i)$, as follows from \cite[Eq. (10)]{BBB} and extracting coefficients:
\begin{equation}
\label{SUMbor}
\sum_{m,n\ge 0}\zt(m+2,\{1\}_n)x^{m+1}y^{n+1}=1-\exp\biggl(\sum_{k\ge 2}\frac{x^k+y^k-(x+y)^k}{k}\zeta(k)\biggr).
\end{equation}
Thus Theorem \ref{thmStirling} implies the conclusion. We remark that an explicit expression for $\zt(m+1,\{1\}_{n-1})$, with $m,n\ge 1$, has been obtained recently by Kaneko and Sakata~\cite{KaSa}:
\[
\zt(m+1,\{1\}_{n-1})=\sum_{i=1}^{\min(m,n)}(-1)^{i-1}
\sum_{\substack{\text{wt}(\mathbf{m})=m,\text{wt}(\mathbf{n})=n 
\\\text{dep}(\mathbf{m})=\text{dep}(\mathbf{n})=i }}
\zt(\mathbf{m}+\mathbf{n}),
\]
where for two indices $\mathbf{m} = (m_1,\dots,m_i)$ and $\mathbf{n} =(n_1,\dots, n_i)$ with weights $\text{wt}(\mathbf{m}) = \sum_{j} m_{j} = m$ and $\text{wt}(\mathbf{n}) = \sum_{j} n_{j} = n$, respectively, which have the same depth $\text{dep}(\mathbf{m})=\text{dep}(\mathbf{n})=i$, the sum $\mathbf{m} + \mathbf{n}$ denotes $(m_1 + n_1,\dots, m_i + n_i)$.
\end{proof}

\subsection{Proofs for the zeta star series\label{sec:zeta_star}}
First we collect as Lemma~\ref{lemBr} a so-called duality result of Hoffman~\cite{H2004}, which generalizes Lemma~\ref{ztsNested}. See also the earlier work of Vermaseren~\cite{V} for an equivalent description in terms of an algorithm, Bradley~\cite{Br} for a $q$-analog, as well as Kawashima~\cite{Ka}. In the following we state Lemma~\ref{lemBr} and give a short and basic proof of it.

\begin{lem}[Nested truncated zeta star values - Duality]
\label{lemBr}
Let $A_{N}^{\ast}(a_1,\dots,a_{q})$ denote an alternating binomial sum of truncated multiple zeta star values:
\[
    A_{N}^{\ast}(a_1,\dots,a_{r})=\sum_{n=1}^{N}\binom{N}{n}\frac{(-1)^{a_1-1}}{n^{a_1}}\zeta^{\ast}_{n}(a_2,\dots,a_r).
\]
Then, for positive integers $N$, $r$, and $a_{i}, b_{i}$, $1 \le i \le r$, it holds
\begin{equation*}
A_{N}^{\ast}(a_1,\{1\}_{b_1-1},\cup_{i=2}^{r}\{a_i+1,\{1\}_{b_i-1}\})
= \zeta_{N}^{\ast}(\cup_{i=1}^{r-1}\{\{1\}_{a_i-1},b_i+1\},\{1\}_{a_r-1},b_r).
\end{equation*}
\end{lem}
\begin{proof}
In order to evaluate $A_{N}^{\ast}(a_1,\{1\}_{b_1-1},\cup_{i=2}^{r}\{a_i+1,\{1\}_{b_i-1})$ we use the strategy of the second proof of Lemma~\ref{ztsNested}. Let 
\[
w_{1,r}=(a_1,\{1\}_{b_1-1},\cup_{i=2}^{r}\{a_i+1,\{1\}_{b_i-1}), 
\]
and $w_{\ell,r}=(\cup_{i=\ell}^{r}\{a_i+1,\{1\}_{b_i-1})$, $2\le \ell\le r$.
Let $k=a_1$. By~\eqref{BradleySimple} we directly get
\begin{align}
A_{N}^{\ast}(w_{1,r}) & = A_{N}^{\ast}(a_{1},\{1\}_{b_{1}-1},w_{2,r}) = \sum_{n=1}^{N}\binom{N}{n}\frac{(-1)^{k-1}}{n^{k}}\zeta^{\ast}_{n}(\{1\}_{b_1-1},w_{2,r}) \notag\\
&=\sum_{i_1=1}^{N}\frac{1}{i_1}\dots \sum_{i_{k}=1}^{i_{k-1}}\frac1{i_{k}}\sum_{n=1}^{i_{k}}\binom{i_{k}}{n}(-1)^{n-1}\zeta^{\ast}_{n}(\{1\}_{b_1-1},w_{2,r}). \label{eqn:ANstar_w1r}
\end{align}
Now we change the order of summation and iterate, similar to equation~\eqref{Iterating}, to get
\begin{align}
\label{eqn:w2r_sum}
&\sum_{n=1}^{i_{k}}\binom{i_{k}}{n}(-1)^{n-1}\zeta^{\ast}_{n}(\{1\}_{b_1-1},w_{2,r})
=\frac{1}{i_k^{b_1-1}}\sum_{j=1}^{i_k}(-1)^{j-1}\binom{i_{k}}{j}\zeta^{\ast}_{j}(w_{2,r}).
\end{align}
If $r=1$, i.e., $w_{2,r}=\emptyset$, then the sum \eqref{eqn:w2r_sum} evaluates to $\frac{1}{i_{k}^{b_{1}-1}}$, and we reobtain Lemma~\ref{ztsNested}:
\begin{equation*}
  A_{N}^{\ast}(w_{1,1}) = \zts_{N}(\{1\}_{a_{1}-1},b_{1}).
\end{equation*}
If $r \ge 2$ then we proceed from \eqref{eqn:w2r_sum} by using $w_{2,r} = a_{2}+1,\{1\}_{b_{2}-1},w_{3,r}$ and again changing summation, which gives:
\begin{align*}
  \frac{1}{i_k^{b_1-1}}\sum_{j=1}^{i_k}(-1)^{j-1}\binom{i_{k}}{j}\zts_{j}(w_{2,r}) = \frac{1}{i_k^{b_1-1}}\sum_{\ell=1}^{i_k}(-1)^{\ell-1}\binom{i_{k}}{\ell}\cdot\frac{\ell}{i_k}\cdot\frac{\zts_{\ell}(\{1\}_{b_{2}-1},w_{3,r})}{\ell^{a_2+1}}.
\end{align*}
Plugging into \eqref{eqn:ANstar_w1r} this implies
\begin{align*}
A_{N}^{\ast}(w_{1,r}) & = A_{N}^{\ast}(a_{1},\{1\}_{b_{1}-1},w_{2,r}) \\
& = \sum_{i_1=1}^{N}\frac{1}{i_1}\dots\sum_{i_{k-1}=1}^{i_{k-2}}\frac1{i_{k-1}}\sum_{i_{k}=1}^{i_{k-1}}\frac1{i_{k}^{b_1+1}} A_{i_k}^{\ast}(a_2,\{1\}_{b_2-1},w_{3,r}).
\end{align*}
From the latter representation, induction with respect to the length $r$ of the argument directly gives the stated result.
\end{proof}

We also collect a well known conversion formula between zeta and zeta star values; see for example Ohno and Zudilin~\cite{Ohno}.
\begin{prop}
Multiple zeta values and multiple zeta star values, as well as the truncated versions, can be expressed
as follows:
\begin{equation}
\label{ztsZeta}
\begin{split}
  \zts(i_1,\dots,i_k) = \sum_{\circ = \text{``},\text{''} \text{or} \, \text{``}+\text{''}} \zt(i_1\circ i_2 \dots \circ i_k),\\
	\zts_{N}(i_1,\dots,i_k) = \sum_{\circ = \text{``},\text{''} \text{or} \, \text{``}+\text{''}} \zt_{N}(i_1\circ i_2 \dots \circ i_k),
\end{split}
\end{equation}
and
\begin{equation}
\label{ztZetaStar}
\begin{split}
  \zt(i_1,\dots,i_k) = \sum_{\circ = \text{``},\text{''} \text{or} \, \text{``}+\text{''}}(-1)^{\sigma_{+}} \zts(i_1\circ i_2 \dots \circ i_k),\\
	\zt_{N}(i_1,\dots,i_k) = \sum_{\circ = \text{``},\text{''} \text{or} \, \text{``}+\text{''}}(-1)^{\sigma_{+}} \zts_{N}(i_1\circ i_2 \dots \circ i_k),
\end{split}
\end{equation}
where $\sigma_{+}$ denotes the number of plus signs in $i_1\circ i_2 \dots \circ i_k$. The sums are taken over all sequences $(\circ_{1}, \dots, \circ_{k-1})$, with $\circ_{i} = \text{``},\text{''} \text{or} \, \text{``}+\text{''}$.
\end{prop}

Now we turn to the proof of Theorem~\ref{thmZetastar}. 
\begin{proof}[Proof of Theorem~\ref{thmZetastar}]
Proceeding exactly as in the proof of Theorem~\ref{thmStirlingFinite}, we get
\[
S=r_1\sum_{k=1}^{r_1}(-1)^{k+1}\binom{r_1-1}{k-1}\sum_{n=1}^{\infty}\frac{\zts_{n-1}(\{1\}_{\ell_1})}{(n+k)n^{r_2}},
\]
and further, after carrying out a partial fraction decomposition, the representation $S = S_{1}^{\ast} + S_{2}^{\ast}$, where
\begin{align*}
S & = r_1\sum_{k=1}^{r_1}(-1)^{k+1}\binom{r_1-1}{k-1}
\bigg[\sum_{m=2}^{r_2}\frac{(-1)^{r_2-m}}{ k^{r_2+1-m}} \sum_{n=1}^{\infty} \frac{\zts_{n-1}(\{1\}_{\ell_1})}{n^m}\bigg]\\
& \quad \mbox{} + r_1\sum_{k=1}^{r_1}(-1)^{k+1}\binom{r_1-1}{k-1}\frac{(-1)^{r_2+1}}{k^{r_2}} \sum_{n=1}^{\infty}\zts_{n-1}(\{1\}_{\ell_1})\Big(\frac{1}{n}-\frac{1}{n+k}\Big).
\end{align*}
Since $\sum_{n \ge 1} \frac{\zts_{n-1}(\{1\}_{\ell})}{n^{m}}$ equals $\zts(m,\{1\}_{\ell})-\zts(m+1,\{1\}_{\ell-1})$, for $\ell \ge 1$, and $\zts(m)$, for $\ell=0$, an application of Lemma~\ref{ztsTrunc} yields the following evaluation of $S_{1}^{\ast}$:
\[
S_{1}^{\ast} = 
\begin{cases}
  \sum_{m=2}^{r_2}(-1)^{r_2-m}\big(\zts(m,\{1\}_{\ell_1})-\zts(m+1,\{1\}_{\ell_1-1})\big)\zts_{r_1}(\{1\}_{r_2-m}), & \quad \text{for $\ell_{1} \ge 1$},\\
	\sum_{m=2}^{r_{2}} (-1)^{r_{2}-m} \zts(m) \zts_{r_{1}}(\{1\}_{r_{2}-m}), & \quad \text{for $\ell_{1}=0$}.
\end{cases}
\]
We turn to an evaluation of $S_{2}^{\ast}$. In the following we write $\ell=\ell_1$ for the sake of simplicity and define 
\[
U^{\ast}(\ell,k) := \sum_{n=1}^{\infty}\zts_{n-1}(\{1\}_{\ell})\Big(\frac{1}{n}-\frac{1}{n+k}\Big), \quad \text{and} \quad T^{\ast}(\ell,k) := \sum_{n=1}^{\infty}\zts_{n}(\{1\}_{\ell})\Big(\frac{1}{n}-\frac{1}{n+k}\Big).
\]
Of course, for $\ell=0$, it holds 
\begin{equation*}
  U^{\ast}(0,k) = T^{\ast}(0,k) = H_{k} = \zt_{k}(1).
\end{equation*}
To get a relation between both series, for $\ell \ge 1$, we write $\zts_{n-1}(\{1\}_{\ell}) = \zts_{n}(\{1\}_{\ell}) - \frac{1}{n} \zts_{n}(\{1\}_{\ell-1})$. After using $\frac{1}{n} \big(\frac{1}{n} - \frac{1}{n+k}\big) = \frac{1}{n^{2}} - \frac{1}{k} \big(\frac{1}{n} - \frac{1}{n+k}\big)$, this simplifies to
\begin{equation*}
  U^{\ast}(\ell,k) = T^{\ast}(\ell,k) + \frac{1}{k} T^{\ast}(\ell-1,k) - \zts(2,\{1\}_{\ell-1}), \quad \text{for $\ell \ge 1$}.
\end{equation*}
Thus, we obtain the following evaluation of the series $S_{2}^{\ast}$ in terms of $T^{\ast}(\ell,k)$:
\begin{equation*}
  S_{2}^{\ast} = 
	\begin{cases}
	  (-1)^{r_2+1}\sum_{k=1}^{r_1}\binom{r_1}{k}\frac{(-1)^{k+1}}{k^{r_2-1}}\Big[T^{\ast}(\ell_{1},k)+\frac1kT^{\ast}(\ell_{1}-1,k)-\zts(2,\{1\}_{\ell_{1}-1})\Big], & \quad \text{for $\ell_{1} \ge 1$},\\
		(-1)^{r_2+1}\sum_{k=1}^{r_1}\binom{r_1}{k}\frac{(-1)^{k+1}}{k^{r_2-1}} \zt_{k}(1), & \quad \text{for $\ell_{1} = 0$}.
	\end{cases}
\end{equation*}

It remains to evaluate $T^{\ast}(\ell,k)$. To this aim we establish for $\ell\ge 1$ a recurrence relation:
\begin{align*}
T^{\ast}(\ell,k)&=\sum_{n=1}^{\infty}\zts_{n}(\{1\}_\ell)\Big(\frac{1}{n}-\frac{1}{n+k}\Big)
=\sum_{j=1}^{\infty}\frac{\zts_{j}(\{1\}_{\ell-1})}j \sum_{n=j}^{\infty} \Big(\frac{1}{n}-\frac{1}{n+k}\Big)\\
&=\sum_{j=1}^{\infty}\frac{\zts_{j}(\{1\}_{\ell-1})}j\Big(H_{j+k-1}-H_{j-1}\Big)=\sum_{j=1}^{\infty}\frac{\zts_{j}(\{1\}_{\ell-1})}j\Big(\sum_{i=0}^{k-1}\frac1{j+i}\Big)\\
&=\sum_{i=1}^{k-1}\frac1i \sum_{j=1}^{\infty}\zeta_{j}^{\ast}(\{1\}_{\ell-1})\Big(\frac1j-\frac1{j+i}\Big)+\zts(2,\{1\}_{\ell-1})
=\sum_{i=1}^{k-1}\frac1i T^{\ast}(\ell-1,i) + t_{\ell},
\end{align*}
with toll function $t_{\ell} = \zts(2,\{1\}_{\ell-1})$, for $\ell \ge 1$.
Iterating this recurrence and taking into account the initial value $T^{\ast}(0,k)=\zeta_{k}(1)$ yields, for $\ell\ge 1$:
\begin{align*}
  T^{\ast}(\ell,k) & = \sum_{j=0}^{\ell-1}\zts(2,\{1\}_{\ell-1-j})\zt_{k-1}(\{1\}_j)+\zt_{k-1}(\{1\}_{\ell+1})+\zt_{k-1}(\{1\}_{\ell-1},2),
\end{align*}
which finishes the proof of the theorem.
\end{proof}

As stated in Remark~\ref{rem:zeta_star}, starting with this representation of the series $S$, it can be evaluated further in terms of truncated multiple zeta star values. Namely, in order to evaluate the finite sums of the type (which occur due to the evaluation of $T^{\ast}(\ell,k)$)
\[
\sum_{k=1}^{r_1}\binom{r_1}{k}\frac{(-1)^{k+1}}{k^{r_2-1}}\zt_{k-1}(a_1,\dots,a_j),
\]
we first convert the truncated zeta values into truncated zeta star values using~\eqref{ztZetaStar}. 
Then, Lemma~\ref{lemBr}, i.e., Hoffman's duality for binomial sums and truncated zeta star values, leads to such an evaluation.

\smallskip

We further remark that an alternative way to get evaluations of $S=S(\ell_{1},0,r_{1},r_{2})$ could be given by starting from \eqref{eqn:sumS} and converting $\zts_{n-1}(\{1\}_{\ell_{1}})$ into ordinary truncated multiple zeta values using \eqref{ztsZeta}.
Then it remains to evaluate general sums of the form 
\begin{equation}
\label{ProKuGeneral}
\sum_{n\ge 1}\frac{\zt_{n-1}(a_1,\dots,a_\ell)}{\binom{n+r_1}{r_1}n^{r_2}}. 
\end{equation}
But this has been carried out explicitly already in~\cite{PK2010}, crucially relying on the result of Hoffman, Lemma~\ref{lemBr}, stated before. However, due to \eqref{ztsZeta}, this leads to a somewhat non-explicit evaluation of $S$. 
Of course, for concrete values of $\ell_{1}$ the evaluation can always be carried out.

\medskip 

\begin{proof}[Proof of Corollary~\ref{corollZTS}]
By a result of Granville~\cite{G} we have
\[
  \zts(2,\{1\}_{\ell-2}) = (\ell-1)\zt(\ell), \quad \text{for $\ell \ge 2$}.
\]
Thus, we can convert all occurrences of zeta star values of this form into ordinary zeta values.
Furthermore, for $r_2=2$ and $\ell_1=2$ we can use 
\[
\zts(3,1)=\frac14\Big(\zts(2,1,1)+2\zt(4)\Big)=\frac14\Big(2\zt(3)+2\zt(4)\Big)
\]
in order to evaluate all occurrences of zeta star values into ordinary zeta values.
\end{proof}

\subsection{Proofs concerning the Arakawa-Kaneko zeta function}

In order to prove Theorem~\ref{thmArakawa} we collect a result of~\cite{K}. 
\begin{lem}[Arakawa-Kaneko zeta function with postive integer arguments]
\label{lemArakawa}
For arbitrary $i_{1},\dots,i_{k},s\in\N$ the function $\xi_{i_{1},\dots,i_{k}}(s)$ is given by
\begin{equation*}
\xi_{i_{1},\dots,i_{k}}(s)=\sum_{n \ge 1}\frac{\zeta_{n}^{\ast}(\{1\}_{s-1}) \, \zeta_{n-1}(i_{2},\dots,i_{k})}{n^{i_{1}+1}}.
\end{equation*}
\end{lem}
\begin{remark}
The special case $k=1$ leads to $\xi_{r}(s) = \zts(r+1,\{1\}_{s-1})$ and was originally obtained by Ohno~\cite{O} using a generalization
of the sum formula for multiple zeta values (see \cite{O} and references therein). However, as mentioned already in this paper, due to a comment by Zagier a simple direct derivation of this result is possible.
\end{remark}

For the sake of completeness we present a short proof of Lemma~\ref{lemArakawa}, essentially by applying Lemma~\ref{ztsTrunc}.
\begin{proof}
We note that by definition 
\[
\Li_{i_1,\dots,i_k}(z)=\sum_{n_1>n_{2}>\dots>n_k\ge 1}\frac{z^{n_1}}{n_1^{i_1}n_2^{i_2}\dots n_{k}^{i_k}}
= \sum_{n=1}^{\infty}\frac{z^n \zt_{n-1}(i_2,\dots,i_k)}{n^{i_1}}.
\]
Thus,
\begin{align*}
\xi_{i_1,\dots,i_k}(s)&=\sum_{n=1}^{\infty}
\frac{\zt_{n-1}(i_2,\dots,i_k)}{n^{i_1}\Gamma(s)}
\int_{0}^{\infty}t^{s-1}e^{-t}(1-e^{-t})^{n-1}dt.
\end{align*}
We already encountered the last integral expression in Remark~\ref{remPWschlkt} when discussing the maximum of standard exponentially distribution random variables, where we obtained, for $s\in\N$, by an application of Lemma~\ref{ztsTrunc}:
\[
\int_{0}^{\infty}t^{s-1}e^{-t}(1-e^{-t})^{n-1}dt
=\frac{1}{n} \E(Z_n^{s-1}) = \frac{(s-1)!}{n}\zts_{n}(\{1\}_{s-1}).
\]
This leads to the stated result.
\end{proof}

\begin{proof}[Proof of Theorem~\ref{thmArakawa}]
We use the recurrence relation for the truncated multiple zeta star values given in Lemma~\ref{ztsTrunc} to get
\[
\zeta_{n_1-1}^{\ast}(\{1\}_{\ell}) = \zeta_{n_1}^{\ast}(\{1\}_{\ell})-\frac1n \zeta_{n_1}^{\ast}(\{1\}_{\ell-1}).
\]
Thus, plugging this relation into the definition of the series $S$ gives, for $r_{1} = 0$, due to Lemma~\ref{lemArakawa} the stated result:
\[
  S(\ell_{1}, \ell_{2}, 0, r_{2}) = \xi_{r_2-1,\{1\}_{\ell_2}}(\ell_1+1) - \xi_{r_2,\{1\}_{\ell_2}}(\ell_1).
\]
On the other hand, combining Lemma~\ref{lemArakawa} with Lemmata~\ref{stirmzv} and~\ref{ztsTrunc} directly leads
to the representation of $\xi_{i_{1},\{1\}_{r-1}}(s)$ involving the polynomials $P_k$ and $Q_k$.
\end{proof}

\section{Outlook and open problems}
\subsection{Evaluation of the general series}
Concerning the general case of evaluating the series $S=S(\ell_{1}, \ell_{2}, r_{1}, r_{2})$ defined in \eqref{eqn:sumS} we may use the proof methods when evaluating the Stirling series and zeta star series as done in Section~\ref{sec:ProofResults}. Proceeding exactly as in the proof of Theorem~\ref{thmStirlingFinite}~\&~\ref{thmZetastar} we get
\[
S=r_1\sum_{k=1}^{r_1}(-1)^{k+1}\binom{r_1-1}{k-1}\sum_{n=1}^{\infty}\frac{\zts_{n-1}(\{1\}_{\ell_1}) \zt_{n-1}(\{1\}_{\ell_2})}{(n+k)n^{r_2}}
\]
as well as the representation $S=S_1+S_2$, via
\begin{align*}
S&=r_1\sum_{k=1}^{r_1}(-1)^{k+1}\binom{r_1-1}{k-1}
\bigg[\sum_{m=2}^{r_2}\frac{(-1)^{r_2-m}}{ k^{r_2+1-m}} \sum_{n=1}^{\infty} \frac{\zts_{n-1}(\{1\}_{\ell_1}) \zt_{n-1}(\{1\}_{\ell_2})}{n^m}\bigg]\\
&+r_1\sum_{k=1}^{r_1}(-1)^{k+1}\binom{r_1-1}{k-1}\frac{(-1)^{r_2+1}}{k^{r_2}} \sum_{n=1}^{\infty}\zts_{n-1}(\{1\}_{\ell_1}) \zt_{n-1}(\{1\}_{\ell_2})\Big(\frac{1}{n}-\frac{1}{n+k}\Big).
\end{align*}
The first series is readily reduced to Arakawa-Kaneko zeta functions by an application of Theorem~\ref{thmArakawa} and Lemma~\ref{ztsTrunc}, which yields
	\[
	S_1=\sum_{m=2}^{r_2}(-1)^{r_2-m}\big(\xi_{m-1,\{1\}_{\ell_2}}(\ell_1+1) - \xi_{m,\{1\}_{\ell_2}}(\ell_1)\big)\zts_{r_1}(\{1\}_{r_2-m}). 
	\]
However, for $S_{2}$ we need to evaluate the series
\begin{equation}\label{eqn:zts_zs_comb}
  \sum_{n=1}^{\infty}\zts_{n-1}(\{1\}_{\ell_1}) \zt_{n-1}(\{1\}_{\ell_2})\Big(\frac{1}{n}-\frac{1}{n+k}\Big).
\end{equation}
For concrete small values of $r_1$ and thus of $k$ this can be done by series manipulation. For symbolic $k$ we may proceed as follows. According to~\eqref{ztsZeta} we can write truncated multiple zeta star values as the sum of truncated multiple zeta values: 
\[
\zts_{n-1}(\{1\}_{\ell_1}) \zt_{n-1}(\{1\}_{\ell_2})
= \sum_{\circ = \text{``},\text{''} \text{or} \, \text{``}+\text{''}} \zt_{n-1}(\{1\}_{\ell_{2}}) \zt_{n-1}(\underbrace{1\circ 1 \dots \circ 1}_{\ell_{1}}).
\]
The so-called stuffle identity for (truncated) multiple zeta values, see, e.g., \cite{H1997}, allows to express products
of (truncated) multiple zeta values 
\[
\zt_N(i_1,\dots,i_k)\zt_N(j_1,\dots,j_\ell)
\]
as the sum of (truncated) multiple zeta values of weight $\sum_{r=1}^{k}i_{r} + \sum_{r=1}^{\ell}j_{r}$ and maximal depth $k+\ell$.
As an example, 
\[
\zt_N(i_1,i_2)\zt(j) = \zt_N(i_1,i_2,j)+\zt_N(i_1,i_2+j)+\zt_N(i_1,j,i_2)+\zt_N(i_1+j,i_2)+\zt_N(j,i_1,i_2).
\]
Using this stuffle identity we obtain from above expression:
\begin{equation}
\label{combination}
\zts_{n-1}(\{1\}_{\ell_1}) \zt_{n-1}(\{1\}_{\ell_2})
= \sum_{\circ = \text{``},\text{''} \text{or} \, \text{``}+\text{''}}\sum_{\mathbf{a}\in\text{stuffle}\big(\{1\}_{\ell_{2}},\underbrace{1\circ 1 \dots \circ 1}_{\ell_{1}}\big)}\zt_{n-1}(\mathbf{a});
\end{equation}
thus $\mathbf{a}=(a_{1},\dots,a_{q})$, for some $q\in\N$.
Hence, an evaluation of the series \eqref{eqn:zts_zs_comb} and thus of $S$ can be reduced to evaluations of series of the form
\[
\sum_{n=1}^{\infty}\zt_{n-1}(a_{1},\dots,a_{q})\Big(\frac{1}{n}-\frac{1}{n+k}\Big),
\]
which have been treated already in~\cite{PK2010}. Altogether, apart from small concrete values of $\ell_1$ and $\ell_2$, this leads to a rather involved expression for the series $S$.

\smallskip 

We also mention a different approach yielding an evaluation of $S$: plugging~\eqref{combination} directly into the definition~\eqref{eqn:sumS} gives 
\[
S=\sum_{\circ = \text{``},\text{''} \text{or} \, \text{``}+\text{''}}\sum_{\mathbf{a}\in\text{stuffle}\big(\{1\}_{\ell_{2}},\underbrace{1\circ 1 \dots \circ 1\big)}_{\ell_{1}}}
\bigg(\sum_{n\ge 1}\frac{\zt_{n-1}(\mathbf{a})}{\binom{n+r_1}{r_1}n^{r_2}}\bigg).
\]
As mentioned before, a treatment of the innermost series~\eqref{ProKuGeneral} has been given already in~\cite{PK2010}; using these evaluations leads to another quite involved expression for $S$, but avoiding Arakawa-Kaneko zeta functions.
We leave as an open problem a general more compact evaluation of the series $S$.

\subsection{Generalized series}
The chosen approach relying on partial fraction decompositions, Lemma~\ref{LemmaKnuth} and \eqref{parfrac}, also applies to series of the form
\begin{equation}
\label{eqn:S_general_k}
\sum_{n\ge 1}\frac{\zts_{n-1}(\{1\}_{\ell_1})\zt_{n-1}(\{1\}_{\ell_2})}{\binom{n+r_1}{r_1}^{k}n^{r_2}},
\end{equation}
for $k\in\N$, generalizing the case $k=1$.
Moreover, a similar approach is applicable to series
\begin{equation}
\sum_{n\ge 1}\frac{\zts_{n-1}(\{1\}_{\ell_1})\zt_{n-1}(\{1\}_{\ell_2})}{\binom{n+r_1}{r_1}\binom{n+r_2}{r_2}n^{r_3}},
\end{equation}
and even to more general products of binomial coefficients. However, it seems that the general case $k\ge 2$ of \eqref{eqn:S_general_k}, even for $\ell_1=0$ or $\ell_2=0$, also involves multiple Hurwitz zeta functions. We leave as an open problem a more compact evaluation of such series.

\subsection{Truncated interpolated multiple zeta values}
We have observed before in Lemmata~\ref{stirmzv} and~\ref{ztsTrunc} that both truncated series, $\zt_{n}(\{1\}_k)$ via the Stirling numbers of the first kind, but also $\zts_{n}(\{1\}_k)$, compare with Remark~\ref{remPWschlkt}, occur in a multitude of different places. 
For non-truncated series $\zt$ and $\zts$, Yamamoto~\cite{Y} introduced a generalization of both versions called interpolated multiple zeta values. 
Namely, according to \eqref{ztsZeta} it holds
\[
\zts(i_1,\dots,i_k)=\sum_{\circ = \text{``},\text{''} \text{or} \, \text{``}+\text{''}}\zt(i_1\circ i_2 \dots \circ i_k).
\]
If we denote by $\sigma_{+}$ the number of plus in the expression $i_1\circ i_2 \dots \circ i_k$, then Yamamoto defines
\begin{equation*}
  \ztt(i_1,\dots,i_k) = \sum_{\circ = \text{``},\text{''} \text{or} \, \text{``}+\text{''}}t^{\sigma_{+}}\zt(i_1\circ i_2 \dots \circ i_k).
\end{equation*}
Thus, the series $\ztt(i_1,\dots,i_k)$ interpolates between multiple zeta values, case $t=0$, and multiple zeta star values, case $t=1$. It turned out that the interpolated series satisfies many identities generalizing or unifying earlier result for multiple zeta and zeta star values.

We are interested in the corresponding truncated series
\begin{align*}
\ztt_{n}(\{1\}_{k}) & =\sum_{\circ = \text{``},\text{''} \text{or} \, \text{``}+\text{''}}t^{\sigma_{+}}\zt_{n}(\underbrace{1 \circ 1 \circ \dots \circ 1}_{k}) = \sum_{\mathbf{p}\in P_O(k)}t^{k-\ell(\mathbf{p})}\zt_{n}(\mathbf{p})
\end{align*}
and their properties, since they interpolate between $\zt_n(\{1\}_k)$ and $\zts_n(\{1\}_k)$, $t=0$ and $t=1$, respectively. 
Here $P_O(k)$ denotes the set of ordered partitions of the integer $k$ and $\ell(\mathbf{p})$ the length of a partition $\mathbf{p}$, defined as the number of its summands. We are looking for interesting (combinatorial) properties of $\ztt_{n}(\{1\}_{k})$ and a possible unification of Lemmata~\ref{stirmzv} and~\ref{ztsTrunc}. A representation in terms of generalized harmonic numbers $H_{n}^{(s)}$ is certainly possible for concrete values of $k$ by a repeated usage of the aforementioned stuffle formulas. 
Below we collect such formulas for the cases $1\le k\le 3$, which are valid for all $t$. 
\begin{example}
For $k=1$ we have 
\[
\ztt_n(\{1\}_1)=\zt_n(1)=\zts(1)=H_n.
\]
For $k=2$, we have
\[
\ztt_n(\{1\}_2)=\frac12\Big[ H_{n}^2+(2t-1)\Ha{n}{2}\Big]
=\frac1{2!}
\left|
\begin{matrix}
H_n & (1-2t)\\
\Ha{n}{2} & H_n
\end{matrix}
\right|.
\]
For $k=3$ we have
\begin{align*}
\ztt_n(\{1\}_3)&=\frac16\Big[H_n^3+(6t-3)H_n\Ha{n}{2}+(6t^2-6t+2)\Ha{n}3\Big]\\
&=\frac1{3!}
\left|
\begin{matrix}
H_n & c_{3,1} &0\\
\Ha{n}{2} & H_n & c_{3,2}\\
\Ha{n}{3} & \Ha{n}{2} & H_n\\
\end{matrix}
\right|,
\end{align*}
with 
\[
c_{3;1}(t), c_{3;2}(t) = -\frac{6t-3}2 \pm\frac{\sqrt{12t^2-12t+1}}{2}. 
\]
\end{example}

We also collect for $t=\frac12$ and $1\le k\le 4$ the specific determinants.
\begin{example}
For $t=\frac12$ and $k=2$ we get
\[
\zt^{\frac12}_n(\{1\}_2)=\frac12 H_{n}^2=
\frac1{2!}
\left|
\begin{matrix}
H_n & 0\\
\Ha{n}{2} & H_n
\end{matrix}
\right|,
\]
and for $k=3$
\[
\zt^{\frac12}_n(\{1\}_3)=\frac16\Big[H_n^3+\frac12\Ha{n}3\Big]
=\frac1{3!}
\left|
\begin{matrix}
H_n & \frac{i}{\sqrt{2}} &0\\
\Ha{n}{2} & H_n & -\frac{i}{\sqrt{2}}\\
\Ha{n}{3} & \Ha{n}{2} & H_n\\
\end{matrix}
\right|,
\]
where $i$ denotes the imaginary unit. For $k=4$ we get
\[
\zt^{\frac12}_n(\{1\}_4)=\frac1{24}\Big[H_n^4+2H_n\Ha{n}3\Big]
=\frac1{4!}
\left|
\begin{matrix}
H_n & 0&0&0\\
\Ha{n}{2} & H_n & i\sqrt{2}&0\\
\Ha{n}{3} & \Ha{n}{2} & H_n&-i\sqrt{2}\\
\Ha{n}{4} & \Ha{n}{3} & \Ha{n}{2}&H_n
\end{matrix}
\right|.
\]
\end{example}

We make the observation, that for $\ztt_{n}(\{1\}_{k})$ no direct generalization of the determinantal representations of $\zt_{n}(\{1\}_{k})$ and $\zts_{n}(\{1\}_{k})$ given in Lemma~\ref{stirmzv} \& \ref{ztsTrunc}, respectively, exist, except for the known cases $t=0$ and $t=1$.
\begin{prop}
For the truncated interpolated multiple zeta values $\ztt_n(\{1\}_k)$ a general determinant of the form
\[
\frac1{k!}\left|
\begin{matrix}
\Ha{n}1 & c_{k,1} & 0 & \dots &0\\
\Ha{n}2 & \Ha{n}1 & c_{k,2} & \dots &0\\
\hdots & \hdots & \hdots & \ddots &\vdots\\
\Ha{n}{k-1} & \Ha{n}{k-2} & \Ha{n}{k-3} & \dots &c_{k,k-1}\\
\Ha{n}{k} & \Ha{n}{k-1} & \Ha{n}{k-2} & \dots &\Ha{n}1\\
\end{matrix}
\right|,
\]
valid for all $k\ge 1$, only exists for the ordinary multiple zeta values $t=0$ and the multiple zeta star values $t=1$.
\end{prop}

\begin{proof}
We will show that for $k=4$, apart from the known cases $t=0$ and $t=1$, only the case $t=\frac12$ leads to such a determinant. Moreover, for $k=5$ even the case $t=\frac12$ does not allow a determinantal representation of this form, which thus proves the stated result.

For $k=4$ we 
get after a repeated use of the stuffle formula eventually
\begin{align*}
\ztt_n(\{1\}_4) & = t^{3} \zt_{n}(4) + t^{2} \big[\zt_{n}(3,1)+\zt_{n}(2,2)+\zt_{n}(1,3)\big]\\
& \quad \mbox{} + t \big[\zt_{n}(2,1,1)+\zt_{n}(1,2,1)+\zt_{n}(1,1,2)\big] + \zt_{n}(\{1\}_{4})\\
& = \frac1{24}\Big(H_n^4 + 6(2t-1)H_n^2\Ha{n}{2} + 8(3t^2-3t+1)H_n\Ha{n}3 + 3(2t-1)^{2}(\Ha{n}2)^2\\
&\quad \mbox{} + 6(2t-1)(2t^{2}-2t+1)\Ha{n}4\Big).
\end{align*}
Assume that
\[
\ztt_n(\{1\}_4)=\frac1{4!}
\left|
\begin{matrix}
H_n & c_{1}(t) &0&0\\
\Ha{n}{2} & H_n & c_{2}(t)&0\\
\Ha{n}{3} & \Ha{n}{2} & H_n&c_{3}(t)\\
\Ha{n}{4} & \Ha{n}{3} & \Ha{n}{2}&H_n
\end{matrix}
\right|.
\]
By Laplace expansion we readily obtain
\begin{equation*}
\frac{1}{24}\big[H_n^4-(c_{1}+c_{2}+c_{3})H_n^2\Ha{n}{2} +(c_{1} c_{2} + c_{2} c_{3})H_n\Ha{n}{3} + c_{1} c_{3} (\Ha{n}{2})^2 - c_{1} c_{2} c_{3} \Ha{n}{4}\big].
\end{equation*}
For $t\neq \frac12$ we obtain, by comparing coefficients of $(\Ha{n}{2})^2$ and $\Ha{n}{4}$, that $c_{2} = -\frac{2(2t^{2}-2t+1)}{2t-1}$; using this and comparing coefficients of $H_n^2\Ha{n}{2}$ and $H_n\Ha{n}{3}$ yields the following necessary equation, which is valid only for the cases $t=0$ and $t=1$:
\[
\frac{8t^{2} (t-1)^2}{(2t-1)(2t^2-2t+1)}=0.
\]
For $t=\frac12$ we get
\[
\zt^{\frac12}_n(\{1\}_4)=\frac1{24}\Big[H_n^4+2H_n\Ha{n}3\Big],
\]
and thus the condition $c_{1} c_{2} c_{3}=0$, with $c_{2}\neq 0$. Choosing $c_{1}=0$ we get
$c_{2} + c_{3}=0$, $c_{2} c_{3}=2$, and thus $c_{2}=i \sqrt{2}$, $c_{3}=-i \sqrt{2}$, leading to the expression stated before.

%

%

%

For $k=5$ we get for $t=\frac12$, after lengthy computations, the result
\[
\ztt_n(\{1\}_5)=\frac1{80}\Ha{n}{5}+ \frac1{24}H_n^2\Ha{n}{3}+\frac1{120}H_n^5.
\]
Expansion of the determinant gives
\begin{align*}
&\frac1{5!}
\left|
\begin{matrix}
H_n & c_{1} &0&0&0\\
\Ha{n}{2} & H_n & c_{2}&0&0\\
\Ha{n}{3} & \Ha{n}{2} & H_n&c_{3}&0\\
\Ha{n}{4} & \Ha{n}{3} & \Ha{n}{2}&H_n&c_4\\
\Ha{n}{5}& \Ha{n}{4} & \Ha{n}{3} & \Ha{n}{2}&H_n\\
\end{matrix}
\right|=\frac1{5!}\bigg(H_n^5-(c_1+c_2+c_3+c_4)H_n^3\Ha{n}{2}\\
&+(c_1c_2+c_2c_3+c_3c_4)H_n^2\Ha{n}{3}
-c_2c_3(c_1+c_4)H_n\Ha{n}{4}\\
& + (c_1c_3+c_1c_4+c_2c_4)H_n(\Ha{n}{2})^2
-c_1c_4(c_2+c_3)\Ha{n}{2}\Ha{n}{3}+c_1c_2c_3c_4\Ha{n}{5}\bigg).
\end{align*}
Comparing coefficients leads to a system of equations for $c_1(t)$, $c_2(t)$, $c_3(t)$, $c_4(t)$ without any solution.
\end{proof}

\medskip

Finally, we mention that for the general formula with $t=\frac12$ we get a truncated analog of the formula of Hoffman and Ihara~\cite[eqn.~$(41)$]{HI}:
\begin{align*}
\zt^{\frac12}_n(\{1\}_k)
&=\sum_{m_1+3m_3+5m_5\dots=k}\frac{2^{m_1+m_3+m_5+\dots}}{2^{n} m_1!m_3!\dots}\Big(\frac{\Ha{n}{1}}{1}\Big)^{m_1} \Big(\frac{\Ha{n}{3}}{3}\Big)^{m_3}\dots\\
&=\frac{1}{2^n(k-1)!}(B_{k-1}(0!\cdot 2\Ha{n}{1},0,2!\cdot 2\Ha{n}{3},0,\dots).
\end{align*}
It can be proven using the algebraic setup of Hoffman and Ihara, which leads to a generating series of the $\zt^{\frac12}_n(\{1\}_k)$. 
Interestingly, all even generalized harmonic numbers $\Ha{n}{2\ell}$ vanish. Their methods also give a formula for arbitrary $t$.

\end{document}